\documentclass[10pt,fleqn]{article}
\usepackage{amssymb}
\usepackage{amsfonts}
\usepackage{amsmath}
\usepackage{amsthm}
\usepackage{graphicx}
\usepackage{epsfig}
\usepackage{psfrag}
\usepackage{color}
\usepackage{dsfont}
\makeatletter
\xdef\@endgadget#1{{\unskip\nobreak\hfil\penalty50\hskip1em\hbox{}\nobreak
    \hfil#1\parfillskip=0pt\finalhyphendemerits=0\par}}
\def\@qedsymbol{${}_\blacksquare$}
\def\qed{\@endgadget{\@qedsymbol}}
\newtheorem{lemma}{Lemma}[section]

\newtheorem{corollary}[lemma]{Corollary}

\newtheorem{definition}[lemma]{Definition}

\newtheorem{proposition}[lemma]{Proposition}
\newtheorem{remark}[lemma]{Remark}
\newcommand{\mR}{\mathbb{R}}

\newcommand{\X}{\mathcal{X}}

\newcommand{\M}{\mathcal{M}}
\newcommand{\R}{\mathcal{R}}
\newcommand{\D}{\mathcal{D}}
\newcommand{\V}{\mathcal{V}}

\newcommand{\E}{\mathcal{E}}
\newcommand{\F}{\mathcal{F}}
\newcommand{\cL}{\mathcal{L}}

\newcommand{\pperp}{\perp \!\!\!\perp}
\newcommand{\bq}{\begin{equation}}
\newcommand{\eq}{\end{equation}}
\newcommand{\bma}{\begin{bmatrix}}
\newcommand{\ema}{\end{bmatrix}}
\DeclareMathOperator{\im}{im}
\DeclareMathOperator{\rank}{rank}

\def\BibTeX{{\rm B\kern-.05em{\sc i\kern-.025em b}\kern-.08em
    T\kern-.1667em\lower.7ex\hbox{E}\kern-.125emX}}

\title{\LARGE \bf Linear port-Hamiltonian DAE systems revisited}

\author{Arjan van der Schaft, Volker Mehrmann
\thanks{A.J. van der Schaft is with the Bernoulli Institute for Mathematics, Computer
Science and AI, Jan C. Willems Center for Systems and Control, University of Groningen, the
Netherlands, {\tt\small A.J.van.der.Schaft@rug.nl},
V. Mehrmann is with the Institut f\"ur Mathematik, TU Berlin, Germany, {\tt\small mehrmann@math.tu-berlin.de}.}
}

\begin{document}

\maketitle
\begin{abstract}
Port-Hamiltonian systems theory provides a systematic methodology for the modeling, simulation and control of multi-physics systems. The incorporation of algebraic constraints has led to a multitude of definitions of port-Hamiltonian differential-algebraic equations (DAE) systems. This paper presents extensions of results in \cite{GerHR21,MehS22_ppt} in the context of maximally monotone structures and shows that any such space can be written as composition of a Dirac and a resistive structure. Furthermore, appropriate coordinate  representations are presented as well as explicit expressions for the associated transfer functions.
\end{abstract}

\noindent


\section{Introduction}
Port-Hamiltonian systems theory provides a systematic methodology for the modeling, simulation and control of large-scale multi-physics systems. It is based on the reticulation of the system into ideal components linked by energy flow. In the finite-dimensional (lumped parameter) case this leads to systems of differential-algebraic equations (DAE systems) with special properties, reflecting the underlying physical structure. These properties can be exploited for purposes of analysis, simulation, and control; see e.g. \cite{SchJ14,MehU23} and the references therein.

The original definition of a port-Hamiltonian system, see e.g. \cite{SchM95,DalS98,SchJ14}, starts from a geometric formulation using Dirac structures (representing the power-conserving interconnection structure), energy-dissipating relations, and Hamilton functions (Hamiltonians) formalizing energy storage, and the algebraic constraints derive from the Dirac structure; see e.g \cite{Sch13}. On the other hand, from a linear DAE point of view the algebraic constraints can be also formulated in terms of the energy relations of the system; see e.g. \cite{BeaMXZ18,MehMW18,MehU23}. Subsequently it was recognized in \cite{SchM18} that geometrically the formulation of algebraic constraints in terms of energy relations corresponds to the replacement of (gradients of) Hamiltonian functions by general Lagrangian subspaces. In the following we will use the terminology \emph{Lagrange structure} to stress the similarity with Dirac structures. Furthermore, in \cite{SchM18} it was shown how algebraic constraints resulting from the Dirac structure can be converted into algebraic constraints corresponding to the Lagrange structure (and conversely) by an augmentation of the state vector.

Various relations between linear port-Hamiltonian DAE formulations emphasizing Dirac structures and those emphasizing Lagrange structures were already studied in \cite{SchM18,GerHR21,MehS22_ppt}.
The present paper extends the results of \cite{GerHR21,MehS22_ppt} in a number of directions, emphasizing the notion of a \emph{maximally monotone structure}. A main result (as conjectured in \cite{MehS22_ppt}) is that any maximally monotone structure can be written as the composition of a Dirac structure and a resistive structure. This leads to an equivalent geometric definition of general linear port-Hamiltonian DAE systems (including energy dissipation and external variables) that is combining maximally monotone subspaces and Lagrange structures. Furthermore, it is shown to provide efficient coordinate representations, as well as explicit expressions for transfer functions.

\noindent
{\bf Notation} We consider finite dimensional real vector spaces and we freely identify linear maps with their matrix representations. Furthermore, given coordinates for some linear space $\V$ we always use corresponding dual coordinates for the dual space $\V^*$.

\section{Linear port-Hamiltonian DAE systems defined by Dirac and Lagrange structures}
In this section we extend the definition of linear port-Hamiltonian DAE systems using Lagrange structures for energy storage, which was originally put forward in \cite{SchM18}, to systems with energy dissipation and external variables.

Consider a finite-dimensional linear space $\F$, with $\dim \F=n$. Consider its dual space $\E:=\F^*$ and denote the duality product by $<e\mid f>, f\in \F, e\in \E$. On the product space $\F \times \E$ define the following two canonical bilinear forms
\[
\begin{array}{l}
\langle (f_1,e_1), (f_2,e_2) \rangle_+ := <e_1 \mid f_2> + <e_2 \mid f_1> \\[2mm]
\langle (f_1,e_1), (f_2,e_2) \rangle_- := <e_1 \mid f_2> - <e_2 \mid f_1>,
\end{array}
\]
with $(f_1,e_1), (f_2,e_2) \in \F \times \E$, corresponding to the two matrix representations
\[
\Pi_+ := \bma 0 & I_n \\ I_n & 0 \ema, \quad \Pi_- :=\bma 0 & -I_n \\ I_n & 0 \ema
\]
(where we recognize $\Pi_-$ as the standard symplectic form).
\begin{definition}
A subspace $\D \subset \F \times \E$ is a \emph{constant Dirac structure} if the bilinear form $\langle \cdot, \cdot \rangle_+$ is zero on $\D$ and moreover $\D$ is maximal with respect to this property. A subspace $\cL \subset \F \times \E$ is a \emph{Lagrange structure} if the bilinear form $\langle \cdot, \cdot \rangle_-$ is zero on $\cL$ and moreover $\cL$ is maximal with respect to this property. A Lagrange structure is called \emph{nonnegative} if the quadratic form defined by $\langle \cdot, \cdot \rangle_+$ is nonnegative on $\cL$.
\end{definition}
From now on \emph{constant} Dirac structures will be simply referred to as \emph{Dirac structures}. Simplest examples of Dirac structures are graphs of \emph{skew-symmetric} maps from $\E$ to $\F$ (or from $\F$ to $\E$). Likewise, Lagrange structures are exemplified by graphs of \emph{symmetric} maps. We have the following equivalent characterizations and properties; see e.g. \cite{SchM18} for the proofs.
\begin{proposition}
A subspace $\D \subset \F \times \E$ is a Dirac structure if and only if $\D=\D^{\pperp_+}$ where ${}^{\pperp_+}$ denotes orthogonal companion with respect to the bilinear form $\langle \cdot, \cdot \rangle_+$. A subspace $\D \subset \F \times \E$ is a Dirac structure if and only $<e \mid f>=0$ for all $(f,e) \in \D$ and $\dim \D= \dim \F$.

A subspace $\cL \subset \X \times \X^*$ is a Lagrange structure if and only if $\cL=\cL^{\pperp_-}$ where ${}^{\pperp_-}$ denotes orthogonal companion with respect to the bilinear form $\langle \cdot, \cdot \rangle_-$. Any Lagrange structure satisfies $\dim \cL =\dim \F$.
\end{proposition}
This leads to the following geometric definition of a (generalized) port-Hamiltonian DAE system, extending the definition given in \cite{SchM18} for systems without energy dissipation and external variables.
\begin{definition}
\label{def:pH1}
Consider the following linear spaces: a state space $\X$, a space of resistive flows $\F_R$, and a space of external flows $\F_P$ (with subscript $P$ referring to 'port'). Furthermore, consider a Dirac structure
\[
\D \subset \X \times \X^* \times \F_R \times \E_R \times \F_P \times \E_P \quad (\mbox{where }\E_R:=\F_R^*,\, \E_P:=\F_P^*),
\]
a Lagrange structure
\[
\cL \subset \X \times \X^* \quad \mbox{ (corresponding to energy \emph{storage})},
\]
and a nonnegative Lagrange structure (called \emph{resistive structure})
\[
\R \subset \F_R \times \E_R \quad \mbox{ (corresponding to energy \emph{dissipation})}.
\]
This defines the port-Hamiltonian DAE system $(\D,\cL,\R)$ as the DAE system
\bq
\begin{array}{rcl}
\{(\dot{x},x,f_R,e_R,f_P,e_P) & \mid &\mbox{there exists } e \in \X^* \mbox{ s.t. } (-\dot{x},e,-f_R,e_R, f_P,e_P) \in \D, \\[2mm]
&& (f_R,e_R) \in \R, \, (x,e) \in \cL \},
\end{array}\label{extdae}
\eq
where $x \in \X, f_R \in \F_R,e_R \in \E_R, f_P \in \F_P,e_P \in \E_P$.
\end{definition}
\begin{remark}
In the 'classical' definition of a port-Hamiltonian DAE system, see e.g. \cite{Sch13,SchJ14}, $\cL= \mbox{graph}(Q):=
\{(x,Qx) \mid x \in \X \}$ for some symmetric matrix $Q: \X \to \X^*$ defining the \emph{Hamiltonian} $\mathcal H(x)=\frac{1}{2}x^\top Q x$. The external variables $e_P,f_P$ are often identified with the inputs and outputs of the system.
Note also that in \cite{MehS22_ppt} systems of the form \eqref{extdae} are called \emph{extended port-Hamiltonian DAE} to distinguish them from the definition of port-Hamiltonian DAEs in \cite{BeaMXZ18}.
\end{remark}
As elaborated in \cite{SchM18}, any Lagrange structure $\cL \subset \X \times \X^*$ can be represented as
\[
\cL = \ker \bma S^\top & - P^\top \ema = \im \bma P \\ S \ema \subset \X \times \X^*
\]
for square matrices $P,S$ satisfying
\bq
\label{SP}
S^\top P= P^\top S, \quad \rank \bma P \\ S \ema =n,
\eq
and, conversely, any such pair $P,S$ defines a Lagrange structure (which is nonnegative if and only if additionally $S^\top P \geq 0$).
It follows that
\bq
\label{SP1}
\cL= \{\bma x \\ e \ema \mid \bma x \\ e \ema = \bma P \\ S \ema z, \; z \in \mR^n \},
\eq
and, by the properties of the Dirac structure $\D$ and the nonnegative Lagrange structure $\R$, the dynamics of the system $(\D,\cL,\R)$ in \eqref{extdae} satisfies the \emph{energy balance}
\bq
\label{eq:energybalance}
\frac{d}{dt} (\frac{1}{2} z^\top S^\top P z) = - <e_R \mid f_R>  + <e_P \mid f_P>\, \leq \, <e_P \mid f_P>.
\eq
Here $<e_P \mid f_P>$ is the \emph{supplied power} due to the external ports with port variables $f_P,e_P$ (e.g., inputs and outputs).
Hence the Hamiltonian $\mathcal H(z):=\frac{1}{2} z^\top S^\top P z$ can be regarded as the \emph{stored energy} of the system (but expressed in the variables $z$ parametrizing $\cL$, instead of the state variables $x$).
Nonnegativity of $\mathcal H(z)$ is equivalent to nonnegativity of the Lagrange structure $\cL$.

Mirroring the developments in \cite{SchM18} we can easily derive \emph{coordinate representations} of linear port-Hamiltonian DAE systems $(\D,\cL,\R)$ as in \eqref{extdae}. Indeed, any Dirac structure $\D$ can be represented as
\[
\D= \{(f,e,f_R,e_R,f_P,e_P) \mid Kf + Le +K_Rf_R +L_Re_R + K_Pf_P +L_Pe_P=0 \}
\]
for certain matrices $K,L,K_R,L_R,K_P,L_P$ satisfying
\bq
\label{KL}
\begin{array}{rcl}
KL^\top + LK^\top + K_RL_R^\top   + L_RK_R^\top   + K_PL^\top_P  +L_PK^\top_P=0 \\[2mm]
\rank \bma K & L & K_R & L_R & K_P & L_P \ema = \dim \left(\X \times \F_R  \times \F_P\right).
\end{array}
\eq
Furthermore, the Lagrange structure $\cL$ can be represented as in \eqref{SP1}, and the resistive structure $\R$ as
\[
\R = \{(f_R,e_R) \mid S^\top_Rf_R= P^\top_Re_R \}, \; S_R^\top P_R= P_R^\top S_R \geq 0, \, \rank \bma P_R \\ S_R \ema = \dim \F_R.
\]
Substituting $f=-\dot{x}=P \dot{z}$ and $e=Sz$ one then obtains the following coordinate representation of a port-Hamiltonian DAE system in the $z$-variables parametrizing $\cL$, with external variables $f_P,e_P$,
\bq
\label{coor1}
KP\dot{z}=LSz +K_Rf_R +L_Re_R + K_Pf_P +L_Pe_P, \quad S^\top_Rf_R= P^\top_Re_R.
\eq
Here $f_R,e_R$ are auxiliary variables.
Note that $z$ is in general \emph{not} the vector of physical state variables $x$; cf. \cite{MehS22_ppt} for an extensive discussion (including physical examples) and an extended equivalence notion. Of course, if $\cL= \mbox{graph}(Q)$ as in the 'classical' definition of a linear port-Hamiltonian DAE system \cite{Sch13,SchJ14}, then $P=I_n$ and $z=x$.

\section{Linear port-Hamiltonian DAE systems defined by maximally monotone and Lagrange structures}
An \emph{alternative} viewpoint on port-Hamiltonian DAE systems is provided by combining the Dirac structure $\D$ and the nonnegative Lagrange structure $\R$ into a maximally monotone structure.
\begin{definition}
A subspace $\M \subset \F \times \E$ is a \emph{monotone structure} if the quadratic form defined by the bilinear form $\langle \cdot, \cdot \rangle_+$ is nonnegative on $\M$. It is a maximally monotone structure if moreover $\M$ is maximal with respect to this property.
\end{definition}
\begin{remark} The definition of a monotone structure is a special case of the notion of a monotone \emph{relation}.
Monotone structures (or subspaces) are also known (after a change of sign) as \emph{ dissipative linear relations}; see e.g. \cite{GerHR21}.
(Nonlinear) port-Hamiltonian DAE systems with respect to a general (maximally) monotone relation were coined as \emph{ incrementally port-Hamiltonian systems} in \cite{CamS13}; see \cite{CamS22} for further developments.
\end{remark}

Some basic properties are collected in the following proposition.
\begin{proposition}
\label{prop:monotone0}
A subspace $\M \subset \F \times \E$ is a monotone structure if and only if
\bq
\label{mono}
<e \mid f>\,  \geq 0
\eq
for all $(f,e) \in \M$, and a maximally monotone structure if $\M$ is maximal with respect to this property. Any monotone structure $\M$ satisfies $\dim \M \leq \dim \F$ and is a maximally monotone structure if and only if $\dim \M = \dim \F$.

Dirac structures and nonnegative Lagrange structures are maximally monotone structures. Furthermore, for any linear map $M: \E \to \F$ with $M+M^\top \geq 0$ its graph $\{(Me,e) \mid e \in \E \}$ is a maximally monotone structure (similarly for $M: \F \to \E$).
\end{proposition}
\begin{proof}
Let $\dim \F =n$. Since $\Pi_+$ has $n$ eigenvalues $+1$ and $n$ eigenvalues $-1$, it follows that the dimension of any monotone structure $\M$ is less than or equal to $n$, and equal to $n$ if and only if $\M$ is a maximally monotone structure. The remaining assertions are obvious.
\end{proof}
From now on a linear map $M$ satisfying $M+M^\top \geq 0$ (or equivalently $e^\top Me \geq 0$ for all $e$) will be called a \emph{ monotone} (or after a change of sign  \emph{ dissipative}; cf. \cite{GerHR21}).
Let us now consider two subspaces with partially shared variables
\[
\M_1 \subset \F_1 \times \E_1 \times \F \times \E, \quad \M_2 \subset \F_2 \times \E_2 \times \F \times \E
\]
Define their \emph{ composition} as
\[
\begin{array}{rcl}
\M_1 \circ \M_2 &:= &\{ (f_1,e_1,f_2,e_2) \in \F_1 \times \E_2 \times \F_2 \times \E_1 \mid \mbox{ there exist } f,e \\[2mm]
&& \mbox{s.t. } (f_1,e_1,f,e) \in \M_1, (f_2,e_2,-f,e) \in \M_2 \}.
\end{array}
\]
It is shown in \cite{MehS22_ppt} that if both $\M_1$ and $\M_2$ are maximally monotone structures then so is $\M_1 \circ \M_2$. Returning to the setting of Definition \ref{def:pH1} it follows by Proposition \ref{prop:monotone} that the \emph{ composition} $\D \circ \R \subset \X \times \X^* \times \F_P \times \E_P$ of the Dirac structure $\D$ with the nonnegative Lagrange structure $\R$ (the resistive structure) is a maximally monotone structure.

This motivates the following {\it alternative} definition of a port-Hamiltonian DAE system, already given in \cite{CamS13} for the nonlinear case (using maximally monotone \emph{relations} instead of subspaces), and in the restricted context of linear DAE systems without external variables $f_P,e_P$ in \cite{GerHR21,MehS22_ppt}.
\begin{definition}
\label{def:pH2}
Consider a linear state space $\X$ and linear spaces of external flow variables $\F_P$  and external effort variables $e_P \in \E_P=\F_P^*$, together with a maximally monotone structure $\M \subset \X \times \X^* \times \F_P \times \E_P$, and a Lagrange structure $\cL \subset \X \times \X^*$. This defines the port-Hamiltonian DAE system $(\M,\cL)$
\[
\{(\dot{x},x,f_P,e_P) \mid \mbox{ there exists } e \in \X^* \mbox{ s.t. } (-\dot{x},e,f_P,e_P) \in \M, (x,e) \in \cL \}.
\]
\end{definition}
\begin{remark} 
The replacement of the composition of a Dirac structure and a resistive structure by a maximally monotone structure may have modeling advantages as well. For instance a lossy transformer may be immediately modeled as a maximally monotone structure, instead of first splitting it into an ideal transformer characteristic and a resistive relation.
\end{remark}
Note that, since the composition of $\D$ with $\R$ is a maximally monotone structure $\M$, Definition \ref{def:pH2} \emph{covers} the previous Definition \ref{def:pH1}. Actually it will turn out that both definitions are \emph{ equivalent}; see Corollary \ref{cor}.
A key observation in proving this (and more) is the following result.
\begin{lemma}
\label{lem:key}
Let $\M \subset \F \times \E$ be a maximally monotone structure. Define
\bq
\label{eq:key}
\begin{array}{rcl}
\F_0&=& \{ f \in \F \mid (f,0) \in \M \}, \\
 \F_1&=& \{ f \in \F \mid \mbox{there exists } e \in \E \mbox{ s.t. } (f,e) \in \M \}, \\[2mm]
 \E_0 &=& \{ e \in \E \mid (0,e) \in \M \}, \\
  \E_1&=& \{ e \in \E \mid \mbox{there exists } f \in \F \mbox{ s.t. } (f,e) \in \M \}.
\end{array}
\eq
Then
\bq
\label{eq:key1}
\F_0 = \E_1^\perp, \quad \E_0 = \F_1^\perp
\eq
\end{lemma}
\begin{proof}
Let $f \in \F_0$. Then $(\alpha f,0) \in \M$ for all $\alpha \in \mathbb{R}$. Consider any $\tilde{e} \in \E_1$, i.e., there exists $\tilde{f}$ such that $(\tilde{f},\tilde{e}) \in \M$. Then by the linearity of $\M$
\[
(\tilde{f} - \alpha f, \tilde{e}) \in \M, \quad \mbox{for all } \alpha\in \mathbb{R}.
\]
This implies $\tilde{e}^\top \tilde{f} \geq \alpha \, \tilde{e}^\top f$ for all $\alpha \in \mathbb{R}$, and thus $\tilde{e}^\top f=0$. Hence $\F_0 \subset \E_1^\perp$. Conversely, let $f \in \E_1^\perp$. Thus, $\tilde{e}^\top f=0$ for all $\tilde{e}$ for which there exists $\tilde{f}$ such that $(\tilde{f},\tilde{e}) \in \M$. Therefore,
\[
(\tilde{e} -0)^\top (\tilde{f} -f) \geq 0
\]
for all $(\tilde{f},\tilde{e}) \in \M$. By the maximality of $\M$ this implies $(f,0) \in \M$, and thus $f \in \F_0$. Hence $\F_0 = \E_1^\perp$. The proof $\E_0 = \F_1^\perp$ is completely analogous.
\end{proof}
\begin{remark}
The same statement as in Lemma~\ref{lem:key} for Dirac structures was proved before in \cite{DalS98}. An analogous result for Lagrange structures follows from the results in \cite{SchM18}.
\end{remark}
The first consequence of this lemma is the following. Recall that, as noted in Proposition \ref{prop:monotone}, $\mbox{graph}(M) \subset \F \times \E$ is a maximally monotone structure for any monotone map $M: \E \to \F$. Conversely, we now show that any maximally monotone structure can be embedded into the graph of a monotone map on an \emph{ augmented} space.
\begin{proposition}\label{prop:monotone}
Consider a maximally monotone structure $\M \subset \F \times \E$, with $\F_0, \F_1, \E_0, \E_1$ as in \eqref{eq:key} in Lemma \ref{lem:key}. Let $\E_1= \ker G^\top$ for some linear map $G: \Lambda \to \X$ and a linear space $\Lambda$. Then there exists a monotone map $M$ such that
\bq
\label{eq:monrep}
\M = \{(f,e) \in \F \times \E \mid \mbox{there exists } \lambda \in \Lambda \mbox{ s.t. } \bma f \\[2mm] 0 \ema = \bma M & -G \\[2mm] G^\top  & 0 \ema \bma e \\[2mm] \lambda \ema \}.
\eq
\end{proposition}
\begin{proof}
Define $M: \E_1 \to \E_1^*$, with $\E_1^*$ the dual space of $\E_1$, as $e \mapsto f$ whenever $(f,e) \in \M$. This is well-defined, since if there exists $f'$ such that also $(f',e) \in \M$, then $f-f' \in \F_0 = \E_1^\perp$, and thus $f$ and $f'$ define the same element of $\E_1^*$. Since $\E_1= \ker G^\top$ it follows that $\F_0 = \E_1^\perp= \im G.$ This corresponds to the addition of vectors $-G \lambda$ with $\lambda \in \Lambda$ arbitrary. It is immediate that $M$ is a monotone map. Finally extend $M$ to a monotone map on the whole of $\E$; again denoted by $M$.
\end{proof}
Note that the augmented linear map
\[
\bma M & -G \\[2mm] G^\top  & 0 \ema : \E \times \Lambda \to \F \times \Lambda^*
\]
is also monotone.
A direct consequence of Proposition \ref{prop:monotone} is the following.
\begin{proposition}
\label{prop:decompose}
Consider a maximally monotone structure $\M \subset \F \times \E$ as in Proposition \ref{prop:monotone}. Then there exist linear spaces $\F_R$ and $\E_R=\F_R^*$ such that $\M = \D \circ \R$ for some  Dirac structure $\D \subset \F \times \E \times \F_R \times \E_R$ and some nonnegative Lagrange structure $\R \subset \F_R \times \E_R$.
\end{proposition}
\begin{proof}
Represent $\M$ as in \eqref{eq:monrep}. The monotone linear map $M$ can be decomposed as $M=-J+R$ with $-J=J^\top$ and $R=R^\top \geq 0$. Then consider flow and effort vectors $f_R \in \F_R:=\F$ and $e_R \in \E_R:=\E$, related by the resistive relation $f_R=Re_R$, defining the nonnegative Lagrange structure $\R= \mbox{graph}(R)$. Define the Dirac structure $\D \subset \F \times \E \times \F_R \times \E_R$ as the linear space of all $f,e,f_R,e_R$ satisfying
\[
f= -Je -G\lambda -f_R, \; 0=G^\top e, \; e_R=e
\]
for some $\lambda \in \Lambda$.
It is immediately checked that $\M = \D \circ \R$.
\end{proof}
Note that Proposition \ref{prop:decompose}, as evidenced by its proof, can be regarded as the 'linear relation' version of the well-known fact that any monotone linear map $M$ can be decomposed as $M=-J +R$, with $-J=J^\top$ the skew-symmetric part of $M$, and $R=R^\top$ the symmetric positive semi-definite part. Proposition \ref{prop:decompose} leads to the following main result of this paper.
\begin{corollary}
\label{cor}
Definition \ref{def:pH2} of a port-Hamiltonian DAE system is equivalent to Definition \ref{def:pH1}.
\end{corollary}
\begin{proof}
We already noticed that Definition \ref{def:pH2} covers Definition \ref{def:pH1} (since the composition of a Dirac structure $\D$ with a nonnegative Lagrange structure $\R$ is a maximally monotone structure $\M$). On the other hand, by Proposition \ref{prop:decompose} any maximally monotone structure $\M \subset \F \times \E \times \F_P \times \E_P$ can be written as $\M = \D \circ \R$ for some Dirac structure $\D \subset \F \times \E \times \F_P \times \E_P \times \F_R \times \E_R$ and nonnegative Lagrange structure $\R \subset \F_R \times \E_R$.
\end{proof}
\begin{remark}
With respect to the proof of Proposition~\ref{prop:decompose} we note that if the matrix $R$ in the decomposition $M=-J+R$ does not have full rank then $\F_R,\E_R$ can be replaced by \emph{ lower-dimensional} spaces $\widetilde{\F}_R,\widetilde{\E}_R$ by first factoring $R:= \widetilde{G}_R \widetilde{R} \widetilde{G}_R^\top$, with $\widetilde{R}=\widetilde{R}^\top >0$. Then we can define the resistive relation as $\widetilde{f}_R=\widetilde{R}\widetilde{e}_R$ resulting in the nonnegative Lagrange structure $\widetilde{\R}$, and define the Dirac structure $\widetilde{\D}$ as the linear space of all $f,e,\widetilde{f}_R, \widetilde{e}_R$ satisfying
\bq
f= -Je -G\lambda -\widetilde{G}_R\widetilde{f}_R, \; 0=G^\top e, \; \widetilde{e}_R=\widetilde{G}_R^\top e
\eq
for some $\lambda$.
\end{remark}

\begin{remark}
We could {\it weaken} the requirements on $\R$ in Definition \ref{def:pH1} by only requiring that $e_R^\top f_R \geq 0$ for all $(f_R,e_R) \in \R$ (or equivalently the quadratic form defined by the bilinear form $\langle \cdot, \cdot \rangle_+$ is nonnegative on $\R$). This would still result in an energy balance as in \eqref{eq:energybalance}.
However, in this case the definition of a port-Hamiltonian DAE system $(\D,\cL,\R)$ is \emph{ not} anymore equivalent to the definition of a port-Hamiltonian DAE system $(\M,\cL)$. See also \cite{GerHR21} for closely related discussions.
\end{remark}
\section{Coordinate representations}
While Definition \ref{def:pH1} gave rise to the coordinate representation \eqref{coor1} of a port-Hamiltonian DAE system $(\D,\cL,\R)$, Definition \ref{def:pH2} leads to the following coordinate representation of a port-Hamiltonian DAE system $(\M,\cL)$. 
By Proposition \ref{prop:monotone} $\M$ can be represented as
\[
\M=\{(f,e,f_P,e_P) \mid \mbox{there exists } \lambda \mbox{ s.t. }\bma f \\[2mm] f_P \\[2mm] 0 \ema = \bma M_{ee} & M_{eP} & - G \\[2mm] M_{Pe} & M_{PP} & -G_P \\[2mm] G^\top & G_P^\top & 0 \ema \bma e \\[2mm] e_P \\[2mm] \lambda \ema \},
\]
for some monotone map $M= \bma M_{ee} & M_{eP} \\ M_{Pe} & M_{PP} \ema$. Decompose $M$ into its skew-symmetric and positive semi-definite symmetric part
\bq
\bma M_{ee} & M_{eP} \\[2mm] M_{Pe} & M_{PP} \ema = \bma -J & -B \\[2mm] B^\top & N \ema + \bma R & V \\[2mm] V^\top & W \ema,
\eq
with $J,N$ skew-symmetric, and $R,W$ symmetric. Substituting $f=-\dot{x}$ then leads to the dynamical equations
\bq
\label{eq:dyna}
\begin{array}{rcl}
\dot{x} & = & \left(J-R\right)e + \left(B-V\right)e_P  + G \lambda \\[2mm]
f_P & = & \left(B + V \right)^\top e + \left(N +W\right)e_P - G_P \lambda \\[2mm]
0 & = & G^\top e + G_P^\top e_P,
\end{array}
\eq
together with $(x,e) \in \cL$.
Then represent the Lagrange structure $\cL$ as in \eqref{SP1}, and thus write $x=Pz$, $e=Sz$. It follows that the dynamics of $(\M,\cL)$ can be written as
\bq
\label{eq:pHnorm2}
\begin{array}{rcl}
\bma P & 0 \\[2mm] 0 & 0 \ema
\bma \dot{z} \\[2mm] \dot{\lambda} \ema & = & \bma J-R & G \\[2mm] G^\top &  0 \ema
\bma S & 0 \\[2mm]  0 & I \ema \bma z \\[2mm] \lambda \ema + \bma B-V \\[2mm] G_P^\top \ema e_P \\[6mm]
f_P & = & \left(B + V \right)^\top Sz + \left(N +W\right)e_P
\end{array}
\eq
Note that system \eqref{eq:pHnorm2} is of the form of a port-Hamiltonian DAE system as introduced e.g. in \cite{BeaMXZ18,MehMW18}.
In particular all algebraic constraints are encapsulated in the {\it augmented} Lagrange structure $\cL_a \subset \X \times \X^* \times \Lambda \times \Lambda^*$ defined by the \emph{augmented} $P$ and $S$ matrices
\[
P_a:= \bma P  & 0 \\[2mm] 0 & 0 \ema, \quad S_a:= \bma S & 0 \\[2mm]  0 & I \ema.
\]
Indeed, the algebraic constraints correspond to singularity of $P_a$. Such algebraic constraints were called \emph{ Lagrange algebraic constraints} in \cite{SchM18}. The augmented Lagrangian structure defines the \emph{degenerate} augmented Hamiltonian
\[
\mathcal H_a(z,\lambda) := \frac{1}{2} \bma z & \lambda \ema^\top S_a^\top P_a \bma z \\[2mm] \lambda \ema = \frac{1}{2} z^\top S^\top P z =\mathcal H(z).
\]
The fact that all algebraic constraints are captured by the Lagrange structure defining energy storage for the system with augmented state vector $(x,\lambda)$ is in line with the result in \cite{SchM18} showing that algebraic constraints arising from the interconnection network structure (called \emph{Dirac algebraic constraints} in \cite{SchM18}) can be converted into Lagrange algebraic constraints by augmenting the state by a Lagrange multiplier vector $\lambda$.

An alternative coordinate representation where not all algebraic constraints are due to the Lagrange structure is the following; extending the developments in \cite{MehS22_ppt}.
Any maximally monotone structure $\M \subset \F \times \E \times \F_P \times \E_P$ can be represented as
\bq
\label{YZ}
\M = \im \bma Z^\top \\ Y^\top \\ Z_P^\top \\ Y_P^\top \ema,
\eq
for matrices $Y,Z,Y_P,Z_P$ satisfying
\bq
\label{YZ1}
YZ^\top + ZY^\top + Y_PZ_P^\top + Z_PY_P^\top\geq 0, \; \rank \bma Z & Y & Z_P & Y_P \ema = \dim \left( \F \times \F_P \right).
\eq
Conversely, any subspace defined by $Y,Z,Y_P,Z_P$ satisfying \eqref{YZ1} is a maximally monotone structure.
This means that any element $(f,e,f_P,e_P) \in \M$ can be represented as
\bq
\label{eq:ZY}
\bma f \\ e \\ f_P \\ e_P \ema = \bma Z^\top \\ Y^\top \\ Z_P^\top \\ Y_P^\top \ema v
\eq
for some $v \in \mR^k$ with $k= \dim \left(\F \times \F_P \right)$. Now consider any matrix tuple $(A,C,A_P,C_P)$ satisfying
\[
\ker \bma A & C & A_P & C_P \ema = \im \bma Z^\top \\ Y^\top \\ Z_P^\top \\ Y_P^\top \ema.
\]
Then pre-multiplication by such a \emph{maximal annihilator} $\bma A & C & A_P & C_P \ema$ eliminates the parametrizing variables $v$, so as to obtain the following kernel representation of $\M$
\[
Af +  Ce +  A_Pf_P + C_Pe_P=0.
\]
Furthermore, representing $\cL$ as in \eqref{SP1}, and substituting $-f=\dot{x}=P\dot{z}$ and $e=Sz$ then yields the DAE system
\bq
\label{ACham}
AP \dot{z} = CSz + A_Pf_P + C_Pe_P,
\eq
where the algebraic constraints are determined both by $A$ and $P$. In fact, the algebraic constraints of \eqref{ACham} can be split, similarly to \eqref{KL}, into two classes: one corresponding to singularity of $P$ (Lagrange algebraic constraints), and one corresponding to singularity of $A$ (algebraic constraints due to the maximally monotone structure $\M$. Mimicking the developments in \cite{MehS22_ppt} one can transform algebraic constraints belonging to one class into algebraic constraints in the other, by the use of additional state variables (Lagrange multipliers).
\begin{remark}
If $\M$ is actually a Dirac structure \emph{ without} flow and effort variables $f_R,e_R$ and represented by $K,L,K_P,L_P$ satisfying \eqref{KL} (without $K_R,L_R$), then we can simply take $\bma A & C & A_P & C_P \ema =\bma K & L & K_P & L_P\ema $, and \eqref{ACham} reduces to the lossless version of \eqref{coor1}.
\end{remark}
Finally note that by \eqref{eq:ZY} the linear space of vectors $\bma e \\e_P \ema$ is equal to $\im \bma Y^\top \\ Y^\top_P \ema$. Hence the matrices $G,G_P$ in \eqref{eq:dyna} are such that (recall \eqref{eq:key})
\[
\ker \bma G^\top & G^\top_P \ema = \im \bma Y^\top \\ Y^\top_P \ema \; = \left(\E \times \E_P\right)_1,
\]
or equivalently
\[
\ker \bma Y & Y_P \ema = \im \bma G \\ G_P \ema \; = \left(\F \times \F_P\right)_0.
\]

\subsection{Transfer functions of port-Hamiltonian DAE systems}
The transfer function of the representation \eqref{eq:pHnorm2} of a port-Hamiltonian DAE system can be derived relatively easily in the special case $G_P=0$. If this condition is \emph{ not} satisfied then it is shown in \cite{MehU23} how by a port-Hamiltonian \emph{ extension} of the state space this condition can be enforced. Note furthermore that geometrically $G_P=0$ is equivalent to the maximally monotone subspace $\M$ of the port-Hamiltonian DAE system being such that
\bq
0 \times \E_P \subset \pi^* (\M),
\eq
where $\pi^*$ is the projection $\pi^*: \F \times \E \times \F_P \times \E_P \to \E \times \E_P$. 

Thus let $G_P=0$, and additionally assume $\cL =\mbox{graph}(Q)$ (i.e., equal to the graph of the gradient of a Hamiltonian $\mathcal H(x)=\frac{1}{2} x^\top Qx$). Then \eqref{eq:pHnorm2} reduces to
\bq
\label{eq:pHnorm}
\begin{array}{rcl}
\dot{x} & = & \left(J-R\right)Qx + \left(B-V\right)u  + G \lambda \\[2mm]
y & = & \left(B + V \right)^\top Qx + \left(N +W\right)u  \\[2mm]
0 & = & G^\top Qx,
\end{array}
\eq
where we have renamed $e_P$ as $u$ (the external variables which are considered to be the \emph{inputs} of the system) and $f_P$ as $y$ (the \emph{outputs} of the system).

Assuming furthermore $Q$ to be \emph{invertible} and, without loss of generality, $G$ to have full column rank, then by differentiating the constraint $G^\top Qx=0$ with respect to time and by substituting the expression for $\dot{x}$, it follows that the Lagrange multiplier vector $\lambda$ is determined as
\[
\lambda = -\left(G^\top Q G\right)^{-1} G_e^\top Q\left(J-R\right)Qx - \left(G^\top Q G\right)^{-1} G^\top Q \left(B-V\right)u.
\]
Substituted into \eqref{eq:pHnorm} this leads to the standard input-state-output system
\bq
\label{eq:explicit}
\begin{array}{rcl}
\dot{x} & = & \Pi \left(J - R\right)  Qx \, + \Pi \left(B-V \right)u, \\[2mm]
y & = & \left(B + V \right)^\top Qx + \left(N +W \right)u,
\end{array}
\eq
where $\Pi:=I - G  \left(G^\top Q G\right)^{-1} G^\top Q$ is a \emph{ projection matrix} (onto $\ker G^\top Q$, orthogonal with respect to the (possibly indefinite) inner product defined by $Q$). It follows that $\ker G^\top Q$ is an invariant subspace of the system \eqref{eq:explicit}.
The transfer function of \eqref{eq:explicit} is
\bq
y = \left(B + V\right)^\top Q \left( sI - \Pi \left(J - R \right)Q\right)^{-1} \Pi \left(B-V\right)u + \left(N + W\right) u,
\eq
which can be checked to be \emph{ positive real} if $Q>0$.

\section{Conclusions}
Although, from a physical systems modeling perspective, the distinction between Dirac structures (power-conserving interconnection) and energy-dissipating relations is often useful, from an analysis and computational point of view the overarching notion of a maximally monotone structure has advantages as well. The current paper contributes to showing the equivalence between these two viewpoints, and in particular provides the useful coordinate representation \eqref{eq:pHnorm2} for any linear port-Hamiltonian DAE system.

\end{document}